\documentclass[a4paper,11pt,oneside,reqno]{amsart}
\usepackage[latin1]{inputenc}
\usepackage[all]{xy}
\usepackage{amssymb, amsmath}
\usepackage{graphicx}

\input xy
\xyoption{all}

\newtheorem{theorem}{Theorem}[section]
\newtheorem{definition}[theorem]{Definition}
\newtheorem{lemma}[theorem]{Lemma}
\newtheorem{corollary}[theorem]{Corollary}
\newtheorem{proposition}[theorem]{Proposition}

\newtheorem{question}{Question}
\newtheorem*{claim*}{Claim}
\newcommand{\nrm}[1]{\|#1\|}
\newcommand{\bignrm}[1]{\Big\|#1\Big\|}

\newcommand{\N}{{\mathbb N}}

\newcommand{\prucl}{\prue[Proof of Claim:]}
\newcommand{\fprucl}{\fprue}
\newcommand{\prue}{\begin{proof}}
\newcommand{\fprue}{\end{proof}}
\newcommand{\conj}[2]{ \{ {#1}\,:\,{#2} \} }
\begin{document}
\title{Banach lattice versions of strict singularity}

\author[J. Flores]{J. Flores}
\address{J. Flores\\ Department of Applied Mathematics, Escet, Universidad Rey Juan Carlos, 28933, M\'ostoles, Madrid, Spain.}
\email{julio.flores@urjc.es}

\author[J. L\'opez-Abad]{J. L\'opez-Abad}
\address{J. L\'opez-Abad \\ Instituto de Ciencias Matematicas (ICMAT).  CSIC-UAM-UC3M-UCM. C/ Nicol\'{a}s Cabrera 13-15, Campus Cantoblanco, UAM
28049 Madrid, Spain} \email{abad@icmat.es}

\author[P. Tradacete]{P. Tradacete}
\address{P. Tradacete \\ Mathematics Department, Universidad Carlos III de Madrid, 28911 Legan\'es, Madrid, Spain.}
\email{ptradace@math.uc3m.es}

\thanks{Research supported by the Spanish Government Grant MTM2012-31286. Third author also supported by Grant MTM2010-14946.}


\begin{abstract}
We explore the relation between lattice versions of strict singularity for operators from a Banach lattice to a Banach space. In particular, we study when the class of disjointly strictly singular operators, those not invertible on the span of any disjoint sequence, coincides with that of lattice strictly singular operators, i.e. those not invertible on any (infinite dimensional) sublattice. New results are given which help to clarify the existing relation between these two classes. 
\end{abstract}

\maketitle

\thispagestyle{empty}

\section{Introduction}

Recall that a linear operator between Banach spaces $T:X\rightarrow Y$ is strictly singular if it is not an isomorphism when restricted to any infinite dimensional subspace. These operators form a two-sided operator ideal, and were introduced by T. Kato \cite{Kato} in connection with the perturbation theory of Fredholm operators. More recently, this class has played an important role in the theory of hereditarily indecompossable spaces.

In this note we explore the relation between lattice versions of strict singularity. Given a Banach lattice $E$ and a Banach space $X$, a bounded operator $T:E\rightarrow X$ is called \emph{disjointly strictly singular} (DSS in short) if there is no infinite sequence $(x_n)_n$ of pairwise
disjoint vectors in $E$ such that the restriction of $T$ to the  span $[x_n]$ is an isomorphism. The class of DSS operators was introduced in \cite{Hernandez-Salinas} in connection with the structure of $\ell_p$ subspaces in Orlicz spaces. Recently, it proved to be a useful tool in the study of strictly singular operators on Banach lattices (cf. \cite{Flores-Hernandez-Kalton-Tradacete}).

Also, $T:E\rightarrow X$ is called \emph{lattice strictly singular} (LSS) if for every infinite-dimensional closed
sublattice $S\subset E$  the restriction $T|_S$ is not an isomorphism. Notice that $T:E\rightarrow X$ is LSS if, and only if, there is no infinite sequence $(x_n)_n$ of \emph{positive}, pairwise
disjoint  vectors in $E$ such that the restriction of $T$ to the span $[x_n]$ is an isomorphism. This follows from the fact that the linear space generated by a sequence of positive disjoint vectors is always a
sublattice, and that every infinite-dimensional sublattice contains a sequence of positive, disjoint vectors.
Note that the linear span of a sequence of (non-positive) disjoint vectors is not a sublattice in general.

The classes of LSS and DSS operators contain all strictly singular operators. An example of a DSS (and LSS) operator which is not strictly singular is the formal identity $i:L_p(\mu)\rightarrow L_q(\mu)$ for $p>q$ and $\mu$ a finite measure: It is easy to see that this operator cannot be an isomorphism on any subspace isomorphic to $\ell_p$, but due to Khintchine's inequality it is an isomorphism on a subspace isomorphic to $\ell_2$.

Clearly, if an operator $T$ is DSS then it is LSS. The main question is whether the converse holds:

\begin{question}\label{mainquestion}
Is every LSS operator also DSS?
\end{question}

A motivation for this question can be traced back to the following observation in
\cite{GJ}: If  an operator $T:E\rightarrow X$ is invertible in a subspace isomorphic to $c_0$ generated by a
disjoint sequence, then $T$ is also invertible in a sublattice isomorphic to $c_0$. In the current
terminology, this is equivalent to saying that, given a compact Hausdorff space $K$ and a Banach space $X$,
every operator $T:C(K)\rightarrow X$ is LSS if and only if it is DSS.

A related question is whether the class of LSS operators forms a subspace of the space of bounded operators between a Banach lattice and a Banach space. More precisely,
\begin{question}\label{questionsum}
Given two LSS operators $T_1,T_2:E \rightarrow X$, is the sum $T_1+T_2$ also LSS?
\end{question}

Notice that the usual proof that shows that the sum of SS operators is SS still works in the class of DSS operators, but it does not in the class of LSS operators. Therefore, a positive answer to Question \ref{mainquestion} would yield an affirmative answer to Question \ref{questionsum}. Surprisingly enough, as will be shown in Theorem \ref{equivalence}, these two questions turn out to be equivalent.

Notice that the class of DSS operators is a left ideal (but not a two-sided ideal). Similarly, the composition of an LSS operator with a bounded operator (from the left) is clearly LSS.

The research on Question \ref{mainquestion} was initiated by the first-named author in \cite{Flores, FH1}, where several partial results for regular operators were given. Recall that an operator between Banach lattices is called regular if it can be written as a difference of two positive operators. Let us summarize these results (\cite[Prop. 4.2]{Flores}, \cite[Prop. 2.4]{FH1}):

\begin{theorem}
Let $T:E\rightarrow F$ be a regular LSS operator between Banach lattices. If any of the following conditions hold, then $T$ is DSS:
\begin{enumerate}
\item $E$ is an L-space and $F$ is a KB-space.
\item $E=L_p(\mu)$ and $F=L_q(\mu)$, $1\leq p\leq \infty$, $1\leq q <\infty$.
\item $T$ is positive and $F$ has order continuous norm.
\end{enumerate}
\end{theorem}

%

Further developments on this question have been done in \cite{FHR}. By means of ultraproduct techniques in combination with Krivine's finite representability Theorem and Brunel-Sucheston's Theorem  on extraction of good sequences, it can be shown that the classes of LSS and DSS operators coincide on a local level. To be more precise, recall that, given a Banach lattice $E$ and a Banach space $X$, an operator $T:E\rightarrow X$  is \emph{super-DSS} (respectively \emph{super-LSS})
if for every $\varepsilon>0$ there exist $N\in\mathbb{N}$ such that for each sequence $(x_n)_{n=1}^N$ of
disjoint elements in $E$ (resp. positive disjoint), there is $x\in[x_n]_{n=1}^N$ with
$\|Tx\|<\varepsilon\|x\|$. It is clear that every super-DSS operator is DSS, and that every super-LSS operator is LSS. Actually, if an
operator $T:E\rightarrow X$ is super-DSS (respectively, super-LSS), then for every free ultrafilter
$\mathcal{U}$, the extension $T_{\mathcal{U}}:E_{\mathcal{U}}\rightarrow X_{\mathcal{U}}$ is DSS (resp. LSS). It was shown in \cite{FHR} that an operator $T$ is super-LSS if and only if it is super-DSS. 

As a consequence, the following was also proved:

\begin{theorem}\label{stable}
Let $E$ be a stable Banach lattice and let $X$ be a stable Banach space with an unconditional finite dimensional decomposition. An operator $T:E\to X$  is LSS if and only if it is DSS.
\end{theorem}

In this note, we show in particular that, the stability hypotheses on $E$ and $X$ can be dropped (see Theorem \ref{main}). The same kind of result will be proved when the range space $X$ is a Banach lattice with non-trivial type, or if $E$ has non-trivial type and $X$ is an order continuous Banach lattice (Corollary \ref{nontrivialtype}).

Moreover, it is shown that the classes of LSS and DSS operators also coincide for operators into stable order continuous Banach lattices. In particular, this holds for operators into spaces of the form $L_1(\mu)$ (Theorem \ref{L1}), which fail to embed in a space with an unconditional basis, and therefore are not covered by our main result (Theorem \ref{main}).

The main difficulty lying under these results is based on the following general question: {\sl suppose that $(x_n)$ is a sequence in a Banach space which can be written as $x_n=y_n+z_n$; what can be said about the properties of $(y_n)$ and $(z_n)$ in terms of those of $(x_n)$?} In particular, regarding the existence of LSS non-DSS operators, we are lead to the problem of finding a certain kind of unconditional blocks in decompositions of an unconditional sequence $(x_n)$. Thus, motivated by Theorem \ref{equivalence}, we will introduce the class of Banach spaces with the \textsl{Unconditional Decomposition Property}, which provide a characterization of when DSS and LSS operators coincide (see Definition \ref{def UDP} and the comments after it).

Throughout we will make repeated use of the classical Kade\v{c}-Pe\l czy\'nski's Dichotomy \cite{Kadec-Pelczynski}: Given an order continuous Banach lattice  $E$, let $i:E\hookrightarrow L_1(\mu)$ denote the canonical inclusion given by \cite[Theorem 1.b.14]{LT2}; for every sequence $(x_n)$ in $E$, either
\begin{enumerate}
\item[(i)] there is $C>0$ such that $\|ix_n\|_1\leq \|x_n\|_E\leq C\|ix_n\|_1$, for every $n\in \mathbb N$; or,
\item[(ii)] there is a disjoint sequence $(y_k)$ in $E$, such that $\|x_{n_k}-y_k\|_E\rightarrow0$.
\end{enumerate}

Our notation is standard: the expression $A\lesssim B$ corresponds to the existence of a constant $C>0$ independent of the parameters involved in $A$ and $B$ such that $A\leq C B$. Also $A\sim B$ will denote $A\lesssim B$ and $B\lesssim A$. Given an element of a Banach lattice $x\in E$, we will consider its positive and negative parts: $x^+=x\vee 0$, $x^-=(-x)\vee 0$ (note that $x=x^+-x^-$ and $|x|=x^++x^-$). The reader is referred to \cite{M-N, LT2} for unexplained notions in Banach lattice theory.

\section{Main results}

Let us start recalling some basic preliminary facts.

\begin{proposition}\label{c0l1}
Let $E$ be a Banach lattice, $X$ a Banach space and $T:E\rightarrow X$ a bounded linear operator.\begin{enumerate}
\item If $T$ is invertible in a subspace isomorphic to $c_0$ generated by a disjoint sequence, then $T$ is also invertible in a sublattice isomorphic to $c_0$.
\item If $T$ is invertible in a subspace isomorphic to $\ell_1$ generated by a disjoint sequence, then $T$ is also invertible in a sublattice isomorphic to $\ell_1$.
\end{enumerate}
\end{proposition}

\begin{proof}
(1): Suppose that $(f_n)$ is a disjoint sequence in $E$ such that $(Tf_n)\sim (f_n)\sim (u_n)_n\text{ in
$c_0$}$. Then there is $\varepsilon\in \{+,-\}$  and an infinite subsequence $(f_n^{\varepsilon})_{n\in M}$
such that $\nrm{Tf_n^\varepsilon}\ge 1/2\nrm{Tf_n}$. Now, given $(a_n)_{n\in M}$ we have that
\begin{align*}
\nrm{\sum_{n\in M}a_nTf_n^\varepsilon}\lesssim\nrm{\sum_{n\in M}a_n f_n^\varepsilon}\lesssim \nrm{\sum_{n\in M}|a_n| |f_n|}
\lesssim\max_{n\in M}|a_n|
\end{align*}
Hence, $(Tf_n^\varepsilon)_{n\in M}$ is a non-trivial weakly-null sequence. It follows from Mazur's Lemma
that there is a basic subsequence $(Tf_n^\varepsilon)_{n\in N}$, hence $(Tf_n^\varepsilon)_{n\in N}$ is
equivalent to the unit vector basis of $c_0$.

(2): Suppose $T$ is invertible on $[x_n]$ where $(x_n)$ are normalized disjoint vectors, and $[x_n]$ is
isomorphic to $\ell_1$. By the uniqueness of unconditional basis of $\ell_1$, it follows that $(x_n)$ must be
equivalent to the unit vector basis of $\ell_1$. Now apply Rosenthal's $\ell_1$-dichotomy to the bounded
sequences $(Tx_n^+)$ and $(Tx_n^-)$. The result follows.
\end{proof}

We will show now that the stability assumptions in Theorem \ref{stable} can be actually dropped. Before, we need the following:

\begin{lemma}\label{LSS}
Let $E$ be a Banach lattice, $X$ a Banach space, and $T:E\to X$ an LSS operator. Suppose that $(x_n)$ is a
normalized sequence of disjoint vectors in $E$ such that $(Tx_n)\sim(x_n)$. Then:
\begin{enumerate}

\item[(a)] Every normalized block sequence of $(x_n)$ with positive coefficients $a_j\geq0$,
$$
y_n=\sum_{j\in A_n} a_j x_j
$$
satisfies
$$
(y_n)\sim(\sum_{j\in A_n} a_j x_j^+)\sim (\sum_{j\in A_n} a_j x_j^-).
$$
\item[(b)] Let $(w_n)_n$ be a positive normalized block subsequence of either $(x_n^+)$ or of $(x_n^-)$. Then
$(T(w_n))$ is seminormalized.

\end{enumerate}
\end{lemma}

\begin{proof}

(a): Let us consider a sequence $(x_n)$ of normalized disjoint vectors in $E$ such that
$$
(Tx_n)\sim(x_n).
$$
Note that both $(x_n)$ and $(Tx_n)$ are unconditional basic sequences.

Now given a normalized block sequence of positive coefficients
$$
y_n=\sum_{j\in A_n} a_j x_j
$$
with $a_j\geq0$, suppose  that $(y_n)\nsim(y_n^+)$. Then, for each $n\in\mathbb{N}$ there would exist $(b_j^n)_{j\in B_n}$ such that
$$
\Big\|\sum_{j\in B_n}b_j^n y_j\Big\|=1
$$
while
$$
\Big\|\sum_{j\in B_n}b_j^n y_j^+\Big\|<\frac1n \hspace{1cm}\textrm{or}\hspace{1cm}\Big\|\sum_{j\in B_n}b_j^n y_j^+\Big\|>n.
$$
Without loss of generality we can assume that the sets $B_n$ are finite and disjoint. Let $w_n=\sum_{j\in B_n}b_j^n y_j$. Notice that $(y_n)$ and $(y_n^+)$ are disjoint sequences, hence unconditional, so we can suppose that $b_j^n\geq0$ for every $n\in\mathbb{N}$ and $j\in B_n$. Hence,
$$
w_n^+=\sum_{j\in B_n}b_j^n y_j^+.
$$
In particular, we have that $\|w_n^+\|\leq \|w_n\|$, so we must have $\|w_n^+\|<1/n$. Therefore,
$$
\|w_n-w_n^-\|=\|w_n^+\|\rightarrow0
$$
and, since $T$ is bounded
$$
\|Tw_n-Tw_n^-\|\rightarrow0.
$$
Hence, since $T$ is an isomorphism on $[y_n]$, and using an argument of perturbation for basic sequences \cite[Proposition 1.a.9]{LT1}, up to further subsequences,  we have the equivalence
$$
(w_n^-)\sim(w_n)\sim(Tw_n)\sim (Tw_n^-)
$$
which yields that $T$ is invertible on the span of $w_n^-$ in contradiction with being LSS.

(b): Suppose otherwise that $(v_n)$ is a positive normalized block subsequence of, say, $(x_n^+)$ such that
$T(v_n)\to 0$. Write $v_n:=\sum_{k\in s_n}a_k x_k^+$, $a_k\ge 0$. Let $w_n:=\sum_{k\in s_n}a_k x_k^-$.
Then, using (a), and the fact that $(T(x_n))\sim (x_n)$,
\begin{align*}
(w_n)\sim&  (v_n-w_n)    \sim (T(v_n-w_n)) \sim (Tw_n),
\end{align*}
and this is impossible.
\end{proof}

Recall that a Banach space $X$ is said to have the unconditional subsequence property (\textsl{USP} in short, cf. \cite{OZ}) if every weakly null sequence has an unconditional subsequence. This is of course the case when $X$ has an unconditional finite dimensional decomposition. In fact, as shown in \cite{JZ}, a strengthening of the \textsl{USP}, namely, the unconditional tree property, is an equivalent condition, for reflexive spaces, to embed in a space with an unconditional basis.

\begin{theorem}\label{main}
Let $E$ be a Banach lattice and $X$ a Banach space with the \textsl{USP}. An operator $T:E\to X$ is LSS if and only if it is DSS.
\end{theorem}

\begin{proof}
Suppose $T$ is not DSS, then there is a disjoint normalized sequence $(x_n)$ satisfying $(Tx_n)\sim (x_n)$. By the previous Lemma, every normalized block sequence $(y_n)$ of $(x_n)$ with positive coefficients satisfies $$(y_n)\sim(y_n^+)\sim(y_n^-).$$
Since $\|x_n\|=1$, we have that both $(x_n^+)$ and $(x_n^-)$ are seminormalized. We will see that $T$ must be invertible on $[x_n^+]$ or $[x_n^-]$.

Indeed, suppose that $(Tx_n^+)\nsim(x_n^+)$. Since $X$ has the \textsl{USP}, up to a further subsequence, we can suppose that $(Tx_n^+)$ is a (seminormalized) unconditional basic sequence. Thus, for every $n\in \mathbb{N}$ there exist $b_j^n\geq0$ with $j\in S_n$ such that
$$
\|\sum_{j\in S_n}b_j^nx_j^+\|=1,
$$
while
$$
\|\sum_{j\in S_n}b_j^nTx_j^+\|<\frac1n, \hspace{1cm}\textrm{or}\hspace{1cm}\|\sum_{j\in S_n}b_j^nTx_j^+\|>n.
$$
Since $T$ is bounded, $\|\sum_{j\in S_n}b_j^nTx_j^+\|>n$ can only hold for finitely many $n$. Therefore, $\|\sum_{j\in S_n}b_j^nTx_j^+\|\rightarrow0$. Hence, there is a subsequence of the block sequence $y_n=\sum_{j\in S_n}b_j^nx_j$ such that
$$
(Ty_n)\sim(Ty_n^-).$$
Since $(y_n)$ has positive coefficients we also have the equivalence $(y_n)\sim(y_n^-)$. Thus, $T$ is invertible on the span of $(y_n^-)$.
\end{proof}

Let us see now that Questions \ref{mainquestion} and \ref{questionsum} are actually equivalent. First, recall that the sum of DSS operators is always DSS. Actually, we have the following:

\begin{lemma}\label{iojio34uori4jffrd}
The sum of an LSS operator with a DSS operator is an LSS operator.
\end{lemma}
\begin{proof}
Fix a Banach lattice $E$, a Banach space $X$, an LSS operator $T:E\to X$, and a DSS operator $S:E\to X$. Let $F
$ be a sublattice of $E$. Since $T$ is LSS, we can find disjoint normalized vectors $(x_n)$ in $F$ such that
$T|_{[x_n]}$ is compact. Now use that $S$ is DSS to find a normalized block subsequence $(y_n)$ of $(x_n)$
where $S|_{[y_n]}$ is also compact. Then $(T+S)|_{[y_n]}$ is compact, so $(T+S)|_F$ is not an isomorphism.
\end{proof}

\begin{theorem}\label{equivalence} 
Let $X$ be a Banach space. The following are equivalent:
\begin{enumerate}
\item[(a)] There is a Banach lattice $E$ and an LSS operator $T:E\to X$ which is not DSS.

\item[(b)] There  are two basic sequences   $(x_n)$ and $(y_n)$ in  $X$ such that:
\begin{enumerate}
\item[(b.1)] No positive block subsequence of $(x_1,y_1,x_2,y_2,\dots)$  is unconditional.
\item[(b.2)] $(x_{n}-y_{n})$ is unconditional.
\item[(b.3)] $\|\sum_n a_n(x_{n}-y_{n})\|\approx\max\{\|\sum_n a_nx_{n}\|,\|\sum_n a_ny_{n}\|\}$.
\end{enumerate}
\item[(c)] There is a basic sequence $(x_n)$ in $X$ such that
\begin{enumerate}
\item[(c.1)] $\nrm{\sum_n |a_n|x_n}\gtrsim \nrm{\sum_n a_n x_n}$.
\item[(c.2)] No positive block subsequence of $(x_n)$ is unconditional (equiv. for every positive block subsequence $(y_n)$ of $(x_n)$
and every $\varepsilon>0$ there is $(b_n)$ such that $\nrm{\sum_n |b_n| y_n}=1$ and $\nrm{\sum_n b_n
y_n}<\varepsilon$).
\item[(c.3)] There is a block subsequence $(z_n)$ of $(x_n)$, $z_n:=\sum_k a_k x_k$, such that
$$\nrm{\sum_{n}b_n z_n}\approx \nrm{\sum_n |b_n|(\sum_k |a_k|x_k)}\approx\max\{\nrm{\sum_n b_n z_n^+},\nrm{\sum_n b_n z_n^-}\},$$
where $z_n^+:=\sum_{a_k>0}a_kx_k$ and $z_n^-:=\sum_{a_k<0}a_kx_k$.
\end{enumerate}
\item[(d)] There is a Banach lattice $E$ and two LSS operators $T_1,T_2:E\to X$ such that $T_1+T_2$ is not
LSS.
\end{enumerate}
\end{theorem}
\begin{proof}
(a) implies (b):  Suppose that $T:E\to X$ is an LSS not DSS operator. Let $(e_n)_n$ be a normalized disjoint
sequence in $E$ such that $(Te_n)_n\sim (e_n)_n$. By Lemma \ref{LSS}, we assume without loss of generality
that $(e_n)_n\sim (e_n^+)_n\sim (e_n^-)_n$. By Proposition \ref{c0l1}, $c_0$ and $\ell_1$ do not embed in
$[e_n]$, hence, since $(e_n)_n$ is unconditional, $[e_n]_n$ is reflexive. It follows that $(Te_n^+)_n$ and
$(Te_n^-)_n$ are weakly-null. By Lemma \ref{LSS} these two sequences are seminormalized. Hence, by passing to some
subsequence and reenumerating if needed, we assume that $(Te_n^+)_n$ and $(T e_n^-)_n$ are
basic sequences. 
Let $x_n=Te_n^+$ and $y_n=Te_n^-$ for every $n\in \N$. We prove that this is the required sequence.  These are two basic
sequences which satisfy (b.1)-(b.3). Indeed, (b.2) holds because $(x_{n}-y_{n})=(T e_n)\sim (e_n)$, which is
unconditional. As for (b.3),
\begin{align*}
\|\sum_n a_n(x_{n}-y_{n})\|\lesssim & \max\{\|\sum_n a_nTe_{n}^+\|,\|\sum_n a_nTe_{n}^-\|\}\lesssim \\
\lesssim & \max\{\|\sum_n a_ne_{n}^+\|,\|\sum_n a_ne_{n}^-\|\}\approx   \|\sum_n a_ne_{n}\| \approx\\
\approx & \|\sum_n a_n(x_{n}-y_{n})\|
\end{align*}
Suppose now that $(v_n)_n$ is a positive block subsequence of $(x_1,y_1,x_2,y_2,\ldots)$  which is
unconditional. Let $w_n=\sum_{k\in s_n} a_k e_{k}^++\sum_{k\in t_n} b_k e_{k}^-$, with $a_k,b_k\ge 0$, be
such that $v_n=Tw_n$. Observe that $s_n\cup t_n < s_{n+1}\cup t_{n+1}$ for every $n$. We use the fact that
$T$ is LSS and that $(v_n)_n$ is unconditional to find a positive normalized block subsequence $(z_n)$ of
$(w_n)$ such that $Tz_n\to 0$. Explicitly, let $z_n:=\sum_{l\in p_n} c_l w_l$, $c_l\ge 0$. Given $n$, and
$k\in s_n\cup t_n$, let
$$d_k:=\left\{\begin{array}{ll}
-a_k & \text{if $k\in s_n\setminus t_n$ or if $k\in s_n\cap t_n$ and $a_n\ge b_n$}\\
b_k & \text{otherwise.}
\end{array}\right.$$
Let
$$w_n':=\sum_{l\in p_n}c_l  \sum_{k\in s_l\cup t_l}d_k e_k.$$
Observe that
\begin{equation*}
w_n'+z_n= \sum_{l\in p_n}c_l \left(\sum_{k\in s_l\setminus t_l}a_k e_k^-+ \sum_{k\in t_l\setminus s_l}b_k e_k^+ +
\sum_{\underset{a_k\ge b_k}{k\in s_l\cap t_l}} (a_k+b_k)e_k^-+  \sum_{\underset{a_k< b_k}{k\in s_l\cap t_l}} (a_k+b_k)e_k^+ \right).
\end{equation*}
is positive for every $n$. Let us verify that $(w_n'+z_n)\sim (w_n')$: Fix scalars $(\lambda_n)_n$. Then, by
the unconditionality of $(e_n)$,
\begin{align}
\Big\|\sum_{n}\lambda_n w_n'\Big\|\approx  \max  & \left\{ \Big\|\sum_n \lambda_n \sum_{l\in p_n}\sum_{k\in s_l\setminus t_l}a_k e_k \Big\|,\Big\|\sum_n \lambda_n \sum_{l\in p_n}\sum_{k\in t_l\setminus s_l}b_k e_k \Big\|
,     \right.\nonumber\\
&\left. \Big\|\sum_n \lambda_n \sum_{l\in p_n}\sum_{k\in s_l\cap t_l, a_k\ge b_k}a_k e_k \Big\|,
 \Big\|\sum_n \lambda_n \sum_{l\in p_n}\sum_{k\in s_l\cap t_l, a_k< b_k}b_k e_k \Big\|  \right\} \label{ojho3irouiodf}
\end{align}
while, also using that $\{e_n^+\}_n\cup \{e_n^-\}_n$ are pairwise disjoint,
\begin{align}
\Big\|\sum_{n}\lambda_n (w_n'+z_n)\Big\|\approx  \max  & \left\{ \Big\|\sum_n \lambda_n \sum_{l\in p_n}\sum_{k\in s_l\setminus t_l}a_k e_k^- \Big\|,
\Big\|\sum_n \lambda_n \sum_{l\in p_n}\sum_{k\in t_l\setminus s_l}b_k e_k^+ \Big\|
,     \right.\nonumber\\
&\left. \Big\|\sum_n \lambda_n \sum_{l\in p_n}\sum_{\underset{a_k\ge b_k}{k\in s_l\cap t_l}}(a_k+b_k) e_k^- \Big\|,
 \Big\|\sum_n \lambda_n \sum_{l\in p_n}\sum_{\underset{a_k< b_k}{k\in s_l\cap t_l}}(a_k+b_k) e_k^+ \Big\|  \right\}\label{ojho3irouiodf1}
\end{align}
and \eqref{ojho3irouiodf} and \eqref{ojho3irouiodf1} are clearly comparable.  Finally, since $T(z_n)\to 0$,
by passing to a subsequence if needed, we have that
\begin{equation}
(T(w_n'+z_n))\sim (T(w_n'))\sim (w_n')\sim (w_n'+z_n).
\end{equation}
This is impossible since $T$ is LSS and $(w_n'+z_n)$ is a disjoint positive sequence.

(b) implies (c):  Fix $(x_n)$ and $(y_n)$ as in (b). Let us see that $(x_n)$ fulfills  the conditions (c.1),
(c.2) and (c.3).
\begin{claim*}
$(x_n)$ and $(y_n)$ are positively equivalent, that is, for every sequence of positive scalars $(a_n)_n$ one
has that $\nrm{\sum_n a_n x_n}\approx \nrm{\sum_n a_n y_n}$.
\end{claim*}
\prucl
Otherwise, we can find a positive   subsequence $(\sum_{k\in s_n}a_k x_k)_n$ such that either it is
normalized and $\nrm{\sum_{k\in s_n}a_k y_k}\le 1/2^n$ for every $n$, or the other way around. Using this and
(b.3), it follows that either $(\sum_{k\in s_n}a_k x_k)_n$ or $(\sum_{k\in s_n}a_k y_k)_n$  are equivalent to
the unconditional basic sequence $(\sum_{k\in s_n}a_k (x_k-y_k))_n$ which is impossible by (b.1), since they
are positive.
\fprucl
(c.1): Fix scalars $(a_n)$. Then,
\begin{align*}
\nrm{\sum_n a_n x_n}\lesssim & \nrm{\sum_n a_n (x_n-y_n)}\approx \nrm{\sum_n |a_n|
(x_n-y_n)} \approx\\
\approx &\max\{\nrm{\sum_n |a_n|x_n},\nrm{\sum_n |a_n|y_n}\} \approx   \nrm{\sum_n |a_n|x_n},\end{align*}
where the last equivalence follows from the previous claim. (c.2) clearly follows from (b.1). We finally
check (c.3). By (b.1), we can find a  normalized block subsequence $(\sum_{k\in s_n}a_k (x_k-y_k))_n$ of
$(x_k-y_k)$ such that $\nrm{\sum_{k\in s_n}a_k y_k}\le 1/2^n $ for every $n$.  It follows that $(\sum_{k\in
s_n}a_k (x_k-y_k))_n$ is equivalent to $(\sum_{k\in s_n}a_k x_k)_n$. Set $z_n:=\sum_{k\in s_n}a_k x_k$ and
$s_n^+:=\conj{k\in s_n}{a_k>0}$ and $s_{n}^-:=s_n\setminus s_n^+$ for every $n$. Hence, given scalars
$(b_n)_n$,
\begin{align*}
\nrm{\sum_n b_n z_n} \approx & \nrm{\sum_n b_n(\sum_{k\in s_n}a_k (x_k-y_k))}\approx  \nrm{\sum_n |b_n|(\sum_{k\in s_n}|a_k| (x_k-y_k))}\approx \\
\approx & \nrm{\sum_n |b_n|(\sum_{k\in s_n}|a_k| x_k)} \ge
\max\{\nrm{\sum_n |b_n|(\sum_{k\in s_n^+}a_k x_k)},\nrm{\sum_n |b_n|(\sum_{k\in s_n^-}a_k x_k)}\}\gtrsim \\
\gtrsim &\max\{\nrm{\sum_n b_n(\sum_{k\in s_n^+}a_k x_k)},\nrm{\sum_n b_n(\sum_{k\in s_n^-}a_k x_k)}\} \gtrsim \nrm{\sum_n b_n z_n}
\end{align*}

(c) implies (d): We fix $(x_n)$ and $(z_n)$ as in (c), $z_n^+:=\sum_{a_k>0} a_k x_k$ and $z_n^-:=\sum_{a_k<0}
a_k x_k$.     Let $E$ be the completion of $c_{00}$ under the norm
$$
\|(a_n)\|_E=\max\{\|\sum_n a_{2n}z_n\|,\|\sum_n a_{2n+1}z_n\|\}.
$$
By (c.3), the unit vector basis $(u_n)$ of $E$ is unconditional, thus, $E$ is an (atomic) Banach lattice. Let
now $T:E\rightarrow X$ be given by $T(u_{2n})=z_n^+$ and $T(u_{2n+1})=-z_n^-$. Using (c.3), we get that
\begin{align*}
\|T(\sum_n a_n u_{2n}+\sum_n b_n u_{2n+1})\|=&  \nrm{\sum_{n}a_nz_n^+-\sum_n b_nz_n^-}\lesssim \max\{\nrm{\sum_{n}a_nz_n^+},\nrm{\sum_n b_nz_n^-}\}\lesssim \\
\lesssim & \max\{\nrm{\sum_{n}a_nz_n},\nrm{\sum_n b_nz_n}\}=\\
=&\nrm{\sum_n a_n u_{2n}+\sum_n b_n u_{2n+1}}.
%
%
\end{align*}
Hence, $T$ is bounded. We see now that $(T(u_{2n}+u_{2n+1}))\sim_1 (u_{2n}+u_{2n+1})$ and consequently $T$ is
not LSS:
\begin{align*}
\nrm{T(\sum_n a_n (u_{2n}+u_{2n+1}))}=&\nrm{\sum_n a_nz_n}=\nrm{\sum_n a_n(u_{2n}+u_{2n+1})}.
\end{align*}
Let $P:E\to [u_{2n}]$ be the Boolean projection, $P(\sum_{n}a_n u_n)=\sum_{n\in 2\mathbb{N}}a_n u_n$, and let
$Q=I-P$. Since $(u_n)$ is unconditional, $P$ and $Q$ are bounded. Let $T_1=T\circ P$, $T_2=T\circ Q$. It is
clear that $T=T_1+T_2$. We see that $T_1$ and $T_2$ are LSS: Suppose otherwise that, e.g. $T_1$ is invertible
in some sublattice of $E$. It follows that there is a positive block subsequence $(v_n)_n$ of $(u_n)_n$  such
that $(T_1v_n)\sim (v_n)$.  Hence, $(TPv_n)$ is an unconditional positive block subsequence of $(x_n)$,
which is impossible by (c.2).

(d) implies (a): Fix a Banach lattice $E$ and LSS operators $T_1,T_2:E\to X$ whose sum is not LSS. It follows
from Lemma \ref{iojio34uori4jffrd} that neither $T_1$ or $T_2$ are DSS.
%
%
\end{proof}

Motivated by the above equivalences we introduce the following definition.

\begin{definition}\label{def UDP}
A Banach space $X$ has the unconditional decomposition property (UDP, in short) if for every unconditional sequence $(z_n)$ in $X$ which can be written as $z_n=x_n+y_n$ for some $(x_n),(y_n)$ in $X$ and such that $$\bignrm{\sum_{n}a_n z_n}\approx\max\Big\{\bignrm{\sum_{n}a_n x_n},\bignrm{\sum_{n}a_n y_n}\Big\},$$ then either $(x_n)$ or $(y_n)$ has an unconditional positive block sequence.
\end{definition}

Theorem \ref{equivalence} yields in particular that a Banach space $X$ has the \textsl{UDP} if and only if for every Banach lattice $E$, every LSS operator $T:E\rightarrow X$ is DSS. Hence, we can reformulate Questions \ref{mainquestion} and \ref{questionsum} into the following

\begin{question}\label{questionUDP}
Does every Banach space have the \textsl{UDP}?
\end{question}

Theorem \ref{main} yields that every Banach space with the \textsl{USP} has the \textsl{UDP}. Trivially, hereditarily indecomposable Banach spaces have the \textsl{UDP}. Note also that having the \textsl{UDP} is an hereditary property: every Banach space which embeds in a Banach space with the \textsl{UDP} has the \textsl{UDP}. It is also easy to check that the direct sum of a finite number of spaces with the \textsl{UDP} also has the \textsl{UDP}. We will see in the next section that Banach lattices with non-trivial type and order continuous stable Banach lattices have the \textsl{UDP}.

%

\section{Operators into Banach lattices}

In this section we focus our analysis on the coincidence of the classes of DSS and LSS operators taking values in a Banach lattice. The fact that the range space is a Banach lattice will allow us to circumvent in some cases the requirement of the unconditional subsequence property in Theorem \ref{main}.

First, let us recall the classical Maurey-Nikishin factorization Theorem (cf. \cite[Theorem 7.18]{albiac-kalton}).

\begin{theorem}\label{nikishin}
Let $X$ be a Banach space of type $r>1$. Suppose that $1\le p<r$ and that $T:X\rightarrow L_p(\mu)$ is an operator. Then $T$ factors through $L_q(\mu)$ for any $p<q<r$. More precisely, for each $p<q<r$ there is a strictly positive density function on $\Omega$ so that $Sx=h^{-\frac{1}{p}}Tx$ defines a bounded operator from $L_p(\mu)$ into $L_q(\Omega, h d\mu)$
\end{theorem}

Note that if $i:L_q(hd\mu)\rightarrow L_p(hd\mu)$ is the natural inclusion and $j:L_p(hd\mu)\rightarrow L_p(\mu)$ is the isometric isomorphism defined by  $j(f)=h^{-\frac{1}{p}}f$, then the above operator $T$ satisfies $T=jiS$ (see \cite[pag. 167]{albiac-kalton}.

Quite unexpectedly the following local property turns out to be the key to avoid the assumption of the USP property in the range space.

\begin{theorem}\label{trivial type}
Let $E$ and $F$ be Banach lattices, with $F$ order continuous. If an LSS operator $T:E\rightarrow F$ is invertible on the span of a disjoint sequence $(x_n)$ in $E$, then $[x_n]$ contains $\ell_1^n$'s uniformly.
\end{theorem}

\begin{proof}
Suppose that $T$ is LSS, and there is a disjoint normalized sequence $(x_n)$ in $E$ such that $(Tx_n)\sim(x_n)$. Let us suppose that $[x_n]$ does not contain $\ell_1^n$'s uniformly, or equivalently, there is $r>1$ such that $[x_n]$ has type $r$.

Note that by Proposition \ref{c0l1}, $[Tx_n]$ does not contain a subspace isomorphic to $\ell_1$ nor $c_0$. Hence, by \cite[Theorem 1.c.5]{LT2}, $[Tx_n]$ is reflexive. By Lemma \ref{LSS}, we have the equivalences $(x_n)\sim (x_n^+)\sim(x_n^-)\sim (Tx_n)$. Also note that no normalized positive block sequence of $(x_n^+)$, i.e
$$
y_n=\sum_{j\in A_n} a_j x_j^+
$$
with $a_j>0$, satisfies $\|Ty_n\|\rightarrow0$. Indeed, if this were the case, then, using Lemma \ref{LSS}, $T$ would be invertible on the span of some subsequence of the sequence $(\sum_{j\in A_n} a_j x_j^-)$, in contradiction with the fact that $T$ is LSS. The same holds for positive blocks of the sequence $(x_n^-)$. In particular, we can assume that $(Tx_n^+)$ and $(Tx_n^-)$ are seminormalized.

Now, since $F$ is order continuous, by \cite[Theorem 1.b.14]{LT2} we can consider the inclusion $i:F\hookrightarrow L_1(\mu)$, for some measure $\mu$. Since, $(x_n)\sim (x_n^+)$, for $q$ satisfying $1<q<r$, by Theorem \ref{nikishin} the restriction of $T$ to the Banach lattice $[x_n^+]$ factors through $L_q(hd\mu)$ with  factors $S, i_h$ and $j$ as above:
$$\xymatrix{[x_n^+]\ar_{S}[d]\ar[r]^T&F\ar@{^{(}->}^{i}[r]&L_1(\mu)\\
L_q(h d\mu)\ar@{^{(}->}^{i_h}[rr]&&L_1(h d\mu)\ar_{j}[u] }$$

If $(Sx_n^+)$ converges to zero in $L_q(hd\mu)$,  then $(iTx_n^+)$ converges to zero in $L_1(\mu)$. Since $(Tx_n^+)$ is seminormalized, by Kade\v{c}-Pe\l czy\'nski's Dichotomy we get that some subsequence $(Tx_{n_k}^+)$ is unconditional. In this case, since $T$ is LSS there exist normalized blocks $(v_j)$ of $(x_{n_k}^+)$, with $\|Tv_j\|\rightarrow0$ and since $(Tx_{n_k}^+)$ is unconditional, we can take this blocks with positive coefficients. This is a contradiction with the assumption above.

Thus, by passing to some subsequence, $(Sx_n^+)$ can be assumed to be seminormalized. Since $q>1$, passing to a subsequence we can assume that $(Sx_n^+)$ is an unconditional basic sequence in $L_q(hd\mu)$. Hence, since $T$ is LSS, we can consider
$$
w_n=\sum_{j\in A_n}a_jx_j^+
$$
such that $\|w_n\|=1$ and $\|Tw_n\|_F\rightarrow0$.

Suppose first that $\|Sw_n\|_q\rightarrow0$. Since $Sx_n^+$ is unconditional, this implies that $\|S|w_n|\|_q\rightarrow0$. Therefore, we get that $\|iT|w_n|\|_1\rightarrow0$, but we know $(T|w_n|)$ cannot go to zero, so by Kade\v{c}-Pe\l czy\'nski's Dichotomy, we get that $(T|w_n|)$ has an unconditional subsequence. This leads to a contradiction, by considering normalized blocks of $(|w_n|)$ whose image under $T$ go to zero, since using the unconditionality of $(T|w_n|)$ these blocks can be taken positive.

Hence, we can assume that $(Sw_n)$ is seminormalized in $L_q$. But since $\|iTw_n\|_1\rightarrow0$, we have that $\|i_hSw_n\|_1\rightarrow0$. Therefore, using Kade\v{c}-Pe\l czy\'nski's Dichotomy, we get a subsequence $(Sw_{n_k})$ which is equivalent to a disjoint sequence in $L_q(hd\mu)$, that is, $(Sw_{n_k})$ is equivalent to the unit vector basis of $\ell_q$. Hence, using that $[x_n^+]$ has type $r$, for every $n\in\mathbb N$ we have
$$
n^{\frac1q}\approx\Big\|\sum_{k=1}^n Sw_{n_k}\Big\|_{L_q(hd\mu)}\leq\|S\|\Big\|\sum_{k=1}^n w_{n_k}\Big\|_E\lesssim n^{\frac1r}
$$
which is a contradiction as $1<q<r$.
\end{proof}

As a consequence of this result we get:

\begin{corollary}\label{nontrivialtype}
Let $E$, $F$ Banach lattices, and $T:E\rightarrow F$ an LSS operator. $T$ is DSS under any of the following:
\begin{enumerate}
\item $F$ has non-trivial type.
\item $E$ has non-trivial type and $F$ is order continuous.
\end{enumerate}
\end{corollary}

Thus, Banach lattices with non-trivial type have the \textsl{UDP}. Nevertheless, non-trivial type is not a necessary condition: the most natural Banach lattices with trivial type, namely $L_1(\mu)$, also have the \textsl{UDP}. This is the content of the next result. Note that, for non-atomic measures $\mu$, the space $L_1(\mu)$ is not isomorphic to a subspace of a space with an unconditional basis (cf. \cite[Proposition 1.d.1]{LT1}). In fact, there exists weakly null sequences in $L_1(\mu)$ without unconditional subsequences \cite{JMS}. Thus, this result does not follow directly from Theorem \ref{main}.

\begin{theorem}\label{L1}
Let $E$ be a Banach lattice  and $T:E\rightarrow L_1(\mu)$ an operator. Then $T$ is LSS if and only if $T$ is DSS.
\end{theorem}

\begin{proof}
Suppose $T$ is LSS but not DSS, then there is a disjoint normalized sequence $(x_n)$ in $E$ such that $(Tx_n)\sim(x_n)$. Note that, by  Theorem \ref{trivial type} $[Tx_n]$ contains $\ell_1^n$'s uniformly. By a well-known result of Rosenthal \cite{Rosenthal} (see also \cite[III.C.Theorem 12]{Woj}), $[Ty_n]$ must contain a subspace isomorphic to $\ell_1$. This is a contradiction by Proposition \ref{c0l1}.
\end{proof}

A combination of this result with Theorem \ref{main} yields that if a Banach lattice has, roughly speaking, many complemented disjoint sequences, then it has the \textsl{UDP}:
\begin{corollary}
Let $F$ be an order continuous Banach lattice such that for every disjoint sequence $(f_n)$ in $F$ there is an infinite dimensional subspace $Y\subset[f_n]$ which is complemented in $F$. Then $F$ has the UDP.
\end{corollary}

\begin{proof}
We use Theorem \ref{equivalence}. Let $T:E\rightarrow F$ be an LSS operator and $(x_n)$ a disjoint sequence in $E$ such that $(Tx_n)\sim(x_n)$. Let $i:F\rightarrow L_1(\mu)$ be the natural inclusion given by \cite[Theorem 1.b.14]{LT2}. By Kade\v{c}-Pe\l czy\'nski's dichotomy either $(iTx_n)\sim (x_n)$ or there is $(y_n)$ in $[Tx_n]$ and a disjoint sequence $(f_n)$ in $F$ such that $\|y_n-f_n\|\rightarrow0$.

Suppose $(iTx_n)\sim (x_n)$, then the operator $iT:E\rightarrow L_1(\mu)$ is an LSS operator which is not DSS. This is a contradiction with Theorem \ref{L1}. Hence, there is $(y_n)$ in $[Tx_n]$ and a disjoint sequence $(f_n)$ in $F$ with $\|y_n-f_n\|\rightarrow0$. Since by assumption $[f_n]$ contains an infinite dimensional complemented subspace, so does $[Tx_n]$. Let $Y$ denote this space and $P:F\rightarrow Y$ the corresponding projection. Now it follows that $PT:E\rightarrow Y$ is an LSS operator which is not DSS, but this is a contradiction with Theorem \ref{main}.
\end{proof}

This applies for instance to every Banach lattice with the positive Schur property, (e.g. Lorentz spaces $\Lambda_\varphi$) or more generally, to every disjointly subprojective Banach lattice (that is, those Banach lattices in which every disjoint sequence contains a complemented subsequence).

Another consequence of Theorem \ref{trivial type} is the following.

\begin{corollary}
Let $E$, $F$ be Banach lattices with $F$ order continuous. If either $E$ or $F$ is stable, then every LSS operator $T:E\rightarrow F$ is DSS.
\end{corollary}

\begin{proof}
Suppose $T:E\rightarrow F$ is an LSS operator which is invertible on the span of some disjoint sequence $(x_n)$. By the stability hypothesis we can assume (passing to further disjoint blocks) that $[x_n]$ is isomorphic to $\ell_p$ for some $p\in[1,\infty)$ \cite{KM}. Using Theorem \ref{trivial type}, it follows that $[x_n]$ contains $\ell_1^n$'s uniformly, hence we must have $p=1$. This is a contradiction according to Proposition \ref{c0l1}.
\end{proof}

In particular, order continuous stable Banach lattices have the \textsl{UDP}. And consequently, by a result of \cite{Garling}, every Orlicz space $L_\varphi$, with an Orlicz function $\varphi$ satisfying the $\Delta_2$ condition, is stable and order continuous, so it has the \textsl{UDP}.

\end{document}